\documentclass[12pt]{amsart}
\usepackage{amsmath,amssymb}
\usepackage{amsthm}
\usepackage{bm}
\usepackage{ascmac}
\usepackage[T1]{fontenc}

\theoremstyle{definition}
\newtheorem{dfn}{Definition}[section]
\newtheorem{thm}[dfn]{Theorem}
\newtheorem{prop}[dfn]{Proposition}
\newtheorem{lem}[dfn]{Lemma}
\newtheorem{ex}[dfn]{Example}

\newtheorem{rem}[dfn]{Remark}

\numberwithin{equation}{section}
\begin{document}
\title{Stability of viscosity solutions on expanding networks}
\author{Shimpei Makida}
\thanks{Department of Mathematics, Graduate School of Science, Hokkaido University, North 10, West 8, Kita-Ku, Sapporo 060-0810, JAPAN}
\email{makida.shimpei.k3@elms.hokudai.ac.jp}
\begin{abstract}
In this paper, we prove the stability of viscosity solutions of the Hamilton--Jacobi equations for a sequence of networks embedded in Euclidean space. The network considered in this paper is not merely a graph---it comprises a collection of line segments. We investigate the conditions under which the stability of viscosity solutions holds if the sequence of networks converges to some compact set in the Hausdorff sense. As a corollary, a characterization of the limit of a sequence of networks on which viscosity solutions can be considered, is obtained. In consideration of this problem, we adopt the concept of viscosity solutions as presented in the sense of Gangbo and \'{S}wi\k{e}ch. 
\end{abstract}
\maketitle

\section{Introduction}
In this paper, we consider Hamilton--Jacobi equations over a sequence of spaces and the corresponding Hamilton--Jacobi equation over the Hausdorff limit of the sequence of spaces.
Let $\{\mathcal{N}^n\}_{n \in \mathbf{N}_{\ge 0}}$ be
expanding networks embedded in $\mathbf{R}^d$ (see Section 2 for the definition) and $\mathcal{N}$ be their limit space (i.e., $\mathcal{N}$ satisfies condition (1) of Theorem \ref{main}).
Let $H$ be a quadratic Hamiltonian of the form
\begin{gather}H:\mathcal{N}\times [0,\infty) \to \mathbf{R},\quad
H(x,p)=\frac{1}{2}|p|^2+V(x) \label{Hamilt}
\end{gather}
with $V \in C(\mathcal{N})$.
We consider a Hamilton--Jacobi equation on $\mathcal{N}$ of the following form:
\begin{gather}
\partial_{t}u(t,x)+H(x,|\nabla u(t,x)|)=0,\,(t,x)\in (0,T)\times \mathcal{N},\label{Shj1} \\
u(0,x)=g(x),\,x \in \mathcal{N}, \label{Shj2}
\end{gather}
where $g: \mathcal{N} \to \mathbf{R}$ is a Lipschitz continuous function.
In contrast, 
we consider a corresponding Hamilton--Jacobi equation on $\mathcal{N}^n$:
\begin{gather}
\partial_{t}u(t,x)+H_{n}(x,|\nabla u(t,x)|)=0,\,(t,x)\in (0,T)\times \mathcal{N}^n, \label{Phj1}\\
u(0,x)=g_{n}(x),\,x \in \mathcal{N}^n,\label{Phj2}
\end{gather}
where $H_{n}=H|_{\mathcal{N}^n}, g_{n}=g|_{\mathcal{N}^n}$.
The purpose of this study was to investigate whether the limit function of  solutions, $u_{n}$, of (\ref{Shj1}), (\ref{Shj2}) is the solution of (\ref{Phj1}), (\ref{Phj2}). Note that it is sufficient for us to consider the case of the Hamiltonian in (\ref{Hamilt}) because its uniqueness allows us to discuss stability similarly. This problem is classified as an investigation of the stability of viscosity solutions with respect to space. The physical motivation lies in the numerical computation of physical phenomena in fractal-like domains by network approximation.

The stability of solutions with respect to space has been studied extensively. In this paper, we focus on fully nonlinear partial differential equations. In \cite{AT2015}, the authors investigated the discounted Hamilton--Jacobi equation for a junction as a limit of its uniformly fattened spaces. In \cite{CCM2016}, the authors investigated the eikonal equation for the Sierpinski gasket as a limit of its prefractals.
In the former paper, the authors demonstrated that the viscosity solutions on fattened spaces converge to the viscosity solution of the effective equation on the junction (i.e., the limiting space) by expressing the viscosity solution of the discounted Hamilton--Jacobi equation in terms of the minimum of the value function obtained via control using ordinary differential equations.
The latter paper demonstrated that the unique viscosity solutions given by the forms of the value functions of the eikonal equation on the graphs converge to the viscosity solution of the eikonal equation on the Sierpinski gasket. The convergence property follows from the monotonicity of the value function.

In this paper, we focus on complicated spaces such as fractals.
In particular, we consider spaces that can be approximated by topological networks, for example, a fractal characterized by contraction maps (see the definition of iterated function system in \cite{F1997,F2014}).
We deal with networks embedded in Euclidean space. A topological network is defined as a merged set obtained by combining the vertex and edge of a graph.
It would be interesting to study equivalent conditions for a space to be approximated by a sequence of graphs, but it is not considered in this paper.
Viscosity solutions of Hamilton--Jacobi equations on networks have been studied by various authors. We refer to \cite{ACCT2013,CS2013,IMZ2013} and the related papers.
In this context, as aforementioned, the problem is reformulated with respect to sequences of networks.

The problem can be considered via two major approaches.
First, the convergence of the value function can be investigated directly (cf. \cite{ACCT2013})and second, the convergence can be shown by a PDE method such as Ascoli-Arzel\`{a} (cf. \cite{CCM2016}). We adopt the PDE approach because it clarifies the argument in future extension. To apply the Ascoli-Arzel\`{a} theorem in our argument, we first prove the uniform Lipschitz property of viscosity solutions. Then, we verify that the obtained limit function is a viscosity solution. We will discuss this in two cases, $\bar{x} \in \bigcup_{l=0}^{\infty}\mathcal{N}^{l}$ and $\bar{x} \in \overline{\bigcup_{l=0}^{\infty}\mathcal{N}^{l}}\setminus \bigcup_{l=0}^{\infty}\mathcal{N}^{l}$. Here, $\bar{x}$ is a touching point of the test function to the limit function.
In the case of $\bar{x} \in \bigcup_{l=0}^{\infty}\mathcal{N}^{l}$, we prove it directly. In the case of $\bar{x} \in \overline{\bigcup_{l=0}^{\infty}\mathcal{N}^{l}}\setminus \bigcup_{l=0}^{\infty}\mathcal{N}^{l}$, we prove it by using Proposition \ref{equivalence}. Proposition \ref{equivalence} is useful for indicating that a function is a viscosity solution. This is because Proposition \ref{equivalence} shows that we only need to consider the test function in the form of the square of the distance function.
$\mathcal{N}$ is a required space on which viscosity solutions can be considered. Based on conditions (1), (2), and (3) of Theorem \ref{main}, which is presented below, we deduce that $\mathcal{N}=\overline{\bigcup_{n=0}^{\infty}\mathcal{N}^{n}}$ and $\mathcal{N}$ is a compact, complete geodesic space. Thus, viscosity solutions in the sense of \cite{GS2015} may be considered on this space.
We also consider the viscosity solutions in each graph \cite{GS2015}. The notion of the viscosity solutions of Hamilton--Jacobi equations in metric spaces was formulated in \cite{AF2014,GHN2015,GS2014,GS2015}---they are primarily of two types, and the relationship between them is considered in \cite{LSZ2021}. Furthermore, the asymptotic behavior of solutions to Hamilton-Jacobi equations was studied in \cite{NN2018}.

The main theorem can be written as follows.
\begin{thm}\label{main}
Let $u_{n}$ be a unique viscosity solution of (\ref{Phj1}), (\ref{Phj2}) for $n=0,1,\cdots$.
Assume the following conditions:
\begin{enumerate}
    \item 
For expanding networks $\{\mathcal{N}^{n}\}_{n \in \mathbf{N}_{\ge 0}}$, there exists a compact set $\mathcal{N}$ such that
\begin{equation*}
d_{H}(\mathcal{N}^{n},\mathcal{N}) \to 0,\quad n \to \infty
,\end{equation*}
where $d_{H}$ is the Hausdorff distance on $\mathbf{R}^{d}$.
\item
Distances ${\delta}_{n}$ and $\tilde{d}$ have the following continuity:
For sequences $a_{n}, b_{n} \in \mathcal{N}^{n}$ converging  $a, b \in \mathcal{N}$ with respect to the distance $\tilde{d}$,
\[\delta_{n}(a_{n}, b_{n}) \to \tilde{d}(a,b), n \to \infty.\]
\item
Let $a_{n}$ and $b_{n}$ be sequences on $\mathcal{N}$. If $\lim_{n \to \infty}d_{E}(a_{n}, b_{n})=0$ on $\mathcal{N}$,
$\lim_{n \to \infty}\tilde{d}(a_{n}, b_{n})=0$,
where $\tilde{d}$ is an intrinsic distance on $\mathcal{N}$ from $\mathbf{R}^d$.
\item
$\tilde{d}$ is bounded on $\mathcal{N} \times \mathcal{N}$ (i.e., $\tilde{d}(a,b) \le C$, $a,b \in \mathcal{N}$ for some positive constant $C$).
\end{enumerate}
Then, there exists a viscosity solution $u:[0,T) \times \mathcal{N} \to \mathbf{R}$ of 
(\ref{Shj1}), (\ref{Shj2}) such that,
on each $[0,T) \times \mathcal{N}^{m}$, $u_{n}$$(n \ge m)$ converges to $u$
uniformly.
\end{thm}
Let us make a few remarks on Theorem \ref{main}. The goal of this study was to construct a viscosity solution based on those in a sequence of finite networks, including fractals. Therefore, the properties of fractals and prefractals characterized by contraction maps (i.e., condition (1) of Theorem \ref{main}) are taken as the basis of the investigation.
Theorem \ref{main} cannot be applied to certain fractals, for example, the Koch curve.
This is because the Koch curve does not admit expanding prefractals.

The remainder of the paper is organized as follows.
In Section 2, we define expanding networks. The definition of the network follows \cite{CS2013}.
In Section 3, we discuss the topology of the limit of expanding networks characterized by the Hausdorff distance.
Section 4 is devoted to proving Theorem \ref{main}. In Section 6, we discuss a supplemental proof of Theorem \ref{main} and further problems.

\section{Definition of expanding networks}
First, we define a network (cf. \cite{CS2013}).
Let $V=\{v_{i}|  i \in I\}$ be a finite set of distinct points in $\mathbf{R}^{d}$. 
Let $\{ \pi_{j}| j\in J\}$ be a finite set of non-self-intersecting line segments satisfying conditions (1)--(4).
For \[\pi_{j}: [0,1] \to \mathbf{R}^{d},\quad j \in J,\]
 let us consider $e_{j}=\pi_{j}((0,1))$, $\bar{e}_{j}=\pi_{j}([0,1])$, and
$E=\{e_{j}| j \in J\}$.
\begin{enumerate}
    \item 
For any $j \in J$, the points $\pi_{j}(0),\pi_{j}(1) \in V.$
    \item
For any $j \in J$, $ \#(\bar{e_{j}}\cap V)=2.$
    \item
For $j, k \in J$, $j\neq k$,
$\bar{e}_{j} \cap \bar{e}_{k} \in V$, $\#(\bar{e}_{j} \cap \bar{e}_{k}) \le 1.$
    \item
For $v, w \in V$, there exists a path connecting $v$ and $w$.
\end{enumerate}
The pair $(V, E)$ is called a graph, and  
$\mathcal{N}=\bigcup_{j \in J}\bar{e}_{j}$
is the network corresponding to $(V,E)$.
Next, we define expanding networks. Let $\{\mathcal{N}^{n}\}_{n \in \mathbf{N}_{\ge 0}}$
be a family of networks
corresponding to a family of graphs $\{(V^{n},E^{n})\}_{n \in \mathbf{N}_{\ge 0}}$. We denote various quantities using the superscript $n$ (e.g., the edge set $E^{n}$ is denoted by $\{ e_{j}^{n}|j\in J^{n}\}$).
\begin{dfn}\label{expanding}
A sequence of networks
$\{\mathcal{N}^{n}\}_{n \in \mathbf{N}_{n \ge 0}}$ is said to be an expanding network if the following conditions are fulfilled.
\begin{enumerate}
\item For any $n \in \mathbf{N}_{\ge 0}$, 
\[\mathcal{N}^{n} \subset \mathcal{N}^{n+1},\quad V^{n} \subset V^{n+1}.\]
\item Under condition (1), the following properties hold for any vertex $v \in \bigcup_{n=0}^{\infty} V^{n}$.
There exist $m \in \mathbf{N}_{\ge 0}$ and $r>0$ such that $v \in V^{m}$ and, for $l \ge m$,
\begin{equation*}
B_{r}(v)\cap \mathcal{N}^{l}
=B_{r}(v)\cap \mathcal{N}^{m},
\end{equation*}
where $B_{r}(v)$ is an open ball with respect to $d_{E}$. By condition (1), $v$ is also a vertex of $V^{l}$.
\end{enumerate}
\end{dfn}
We now provide a simple example of expanding networks.
\begin{ex}For $i \in \mathbf{N}$, set
\[v_{i}=\sum_{k=1}^{i}\frac{1}{2^{k}} \in \mathbf{R}\]
and define $V^{n}=\{v_{i}|i=0,1,\cdots, n\}$, where $v_{0}=0$. Corresponding to $V^{n}$, we define $E^{n}$ as all segments that connect $v_{i}$ and $v_{i+1}$ for $0 \le i < n$. 
We consider $\mathcal{T}^{n}$ as the corresponding network of the family of graphs $\{(V^{n},E^{n})\}_{n \in \mathbf{N}_{\ge 0}} $.  
$\mathcal{T}^{n}$ satisfies the assumptions of Theorem \ref{main}, and their Hausdorff limit is $[0,1]$ in $\mathbf{R}$.
\end{ex}
We can treat the Sierpinski gasket as an example of expanding networks if we replace $r$ with $r_{l}$ in condition (2) of Definition \ref{expanding}.
\begin{ex}\label{sierpinski}
Let $\Gamma=\{a_{1},a_{2},a_{3}\} \subset \mathbf{R}^2$ be the set of vertices of an equilateral triangle in $\mathbf{R}^2$. We define a contraction map $\psi_{i}:\mathbf{R}^2 \to \mathbf{R}^2$ $(i=1,2,3)$ by
\[\psi_{i}(x)=a_{i}+\frac{1}{2}(x-a_{i}).\]
As ${V}^0=\Gamma$, let us define vertex sets $V^{n}$ as follows:
\begin{equation*}
{V}^n=\displaystyle\bigcup_{1\le i_{1},i_{2},\cdots,i_{n}\le 3}\psi_{i_{n}}\circ \cdots \psi_{i_{1}}(\Gamma).
\end{equation*}
Furthermore, we define ${E}^n$ to be the set of all open line segments whose endpoints are $x,y \in {V}^n$ and that are images of $n$-compositions by $\psi_{i_k}$ $(k=0,1,\cdots,n)$ of one side connecting two points of $\Gamma$. Using ${V}^n$ and ${E}^n$, as constructed above, we define the prefractal $\mathcal{S}^n={V}^n\cup {E}^n,n\ge0$ as a network. We remark that $\{\mathcal{S}^{n}\}_{n \in \mathbf{N}_{\ge 0}}$ are expanding networks. It is known that the Sierpinski gasket $\mathcal{S}$
can be constructed by $\mathcal{S}=\overline{\bigcup_{n=0}^{\infty}\mathcal{S}^n}$ (cf. \cite{F1997}).
\end{ex}

\section{Generalization to metric spaces}
In this section, we investigate metric space properties of the limit, $\mathcal{N}$, of expanding networks satisfying the assumption of Theorem \ref{main}.
Let $(X,d)$ be a metric space. Let $a,b \in \mathbf{R}$ and $a<b$.
\begin{dfn}
The length $L(\gamma)$ of a continuous curve $\gamma : [a,b] \to X$ is defined as follows:
\begin{equation*}
L(\gamma):=\sup_{\sigma} \sum_{i=1}^{n-1}d(\gamma(t_{i}),\gamma(t_{i+1})),
\end{equation*}
where $\sigma$, on the right-hand side, moves over the entire partition of $[a,b]$.
We say that a continuous curve $\gamma$ is a rectifiable curve if $\gamma : [a, b]\to X$ satisfies $L(\gamma)<\infty$.

A metric space $(X, d)$ is rectifiably connected if $(X,d)$ satisfies the following properties.
For all $x,y \in X$, there exists a rectifiable curve $\gamma$ satisfying $\gamma(a)=x$ and $\gamma(b)=y$.
A rectifiable curve $\gamma$ satisfying this property is called a rectifiable curve connecting $x$ and $y$.

A rectifiable connected metric space $(X,d)$ is called a length space if $(X,d)$ satisfies the following.
For $x,y \in X$,
\[d(x,y)=\inf_{\gamma}L(\gamma),\]
where $\gamma$ moves over the entire rectifiable curve connecting $x$ and $y$.
\end{dfn}
\begin{rem}
For any $x,y$ in each network $\mathcal{N}^{n}$, we define the distance between them to be
\[\delta_{n}(x,y)=\inf_{\gamma}L(\gamma),\]
where $\gamma$ moves over the entire rectifiable curve in $\mathcal{N}^{n}$ connecting $x$ and $y$.
Then, $(\mathcal{N}^{n},\delta^{n})$ is a compact geodesic space.
\end{rem}
Next, we introduce the Hausdorff distance to discuss the convergence of a family of subsets.
Let $(X,d)$ be a distance space.
Consider $A \subset X$. We define the distance of a point $x \in X$ from set $A$, denoted by $d_{A}$, as follows:
\[ d_{A}(x)=\inf_{y \in A}d(y,x), x \in X. \] 
Using $d_{A}$, we define an $\epsilon$-neighborhood of $A$ as follows: \[N(A,\epsilon)=\{x \in X|d_{A}(x)\le \epsilon\}.\]
\begin{dfn}
For $A,B \subset X$, the Hausdorff distance $d_{H}(A,B)$ is defined as follows:
\[d_{H}(A,B)=\inf\{\epsilon \ge 0|B \subset N(A,\epsilon),
A \subset N(B,\epsilon)\}. \]

Let $\{A_{n}\}_{n=1}^{\infty}$ be a family of subsets of $X$.
We define the closed upper and lower limits of $\{A_{n}\}_{n=1}^{\infty}$ as follows:
\begin{equation*}
\limsup_{n \to \infty}A_{n}=\{ x \in X| \forall \epsilon>0, \# \{k:B_{\epsilon}(x) \cap A_{k}\neq \emptyset\}=\infty \},
\end{equation*}
\begin{equation*}
\liminf_{n \to \infty}A_{n}=
\{ x \in X| \forall \epsilon>0, \#\{k:B_{\epsilon}(x) \cap A_{k}=\emptyset\}<\infty \}.
\end{equation*}

The closed limit $\lim_{n \to \infty}A_{n}$ is defined to be 
\[\lim_{n \to \infty}A_{n}=\liminf_{n \to \infty}A_{n}=\limsup_{n \to \infty}A_{n},\]
if $\liminf_{n \to \infty}A_{n}=\limsup_{n \to \infty}A_{n}$ holds.
\end{dfn}
The following proposition reveals the relationship between the Hausdorff distance and the limit of a family of sets.
\begin{prop}[{\cite[Proposition 4.3.5]{P2005}}]\label{hau.conv.}
Let $\{A_{n}\}_{n=1}^{\infty}$ be a family of subsets of $X$ and $A$ be a bounded subset of $X$.
If
\[d_{H}(A_{n},A) \to 0,\quad n\to \infty,\]
then
\[\lim_{n \to \infty}A_{n}=\overline{A}.\]
\end{prop}
\begin{prop}
If $\mathcal{N}$ and $\{\mathcal{N}^{n}\}_{n \in \mathbf{N}_{\ge 0}}$ satisfy conditions (1), (3), and (4) of Theorem \ref{main}, then $\mathcal{N}$ is a compact geodesic space with respect to the intrinsic distance.
\end{prop}
\begin{proof}
First, the intrinsic distance on $\mathcal{N}$ can be defined by
\[\tilde{d}(x,y)=\inf_{\gamma}L(\gamma),\quad x,y \in \mathcal{N}, \]
where $\gamma$ moves over the entire rectifiable curve connecting $x$ and $y$.
By the definition of $\mathcal{N}$ and (4) in Theorem \ref{main}, $\mathcal{N}$ is a geodesic space.
See Lemma 2.11. in \cite[Lemma 2.11]{B1998}.

Next, we prove that the distance space $(\mathcal{N},\tilde{d})$ is compact.
For this purpose, it is sufficient to prove that 
$(\mathcal{N},\tilde{d})$ is sequentially compact.
Let $\{x_{n}\}_{n \in \mathbf{N}}$ be a sequence of points in $\mathcal{N}$.
Because $\mathcal{N}$ is compact with respect to $d_{E}$,
there exists a sequence $\{n_{j}\}_{j \in \mathbf{N}}$ and a point $x\in \mathcal{N}$ such that
\[d_{E}(x_{n_{j}},x) \to 0,\quad j \to \infty. \]
Using condition (3), we obtain
\[\tilde{d}(x_{n_{j}},x)\to 0,\quad j \to \infty.\]
This proves that $(\mathcal{N}, \tilde{d})$ is compact.

\end{proof}
\section{Viscosity solutions on metric spaces}
Let $(\Omega, d)$ be a complete geodesic space.
Let $x_{0} \in \Omega$ be a fixed point.
In this paper, we use the formulation of viscosity solutions introduced by Gangbo and \'{S}wi\k{e}ch in \cite{GS2015}.
For the sake of convenience, we introduce the viscosity solution of a Hamilton--Jacobi equation with a Hamiltonian of the same type as in (\ref{Hamilt}) and investigate its properties.

An analogous notion for a derivative is proposed for a function on a metric space. Assume $T > 0$. For $v :(0,T)\times \Omega \to \mathbf{R}$, the upper and lower slopes of $v$ at $(t,x) \in (0,T)\times \Omega$ are defined, respectively, by
\begin{gather*}
|\nabla^{+}v(t,x)|= \limsup_{y \to x}\frac{[v(t,y)-v(t,x)]_{+}}{d(y,x)},\\ |\nabla^{-}v(t,x)|= \limsup_{y \to x}\frac{[v(t,y)-v(t,x)]_{-}}{d(y,x)}.
\end{gather*}
The slope of $v$ at $(t,x) \in (0,T)\times \Omega$ is defined by
\[|\nabla v(t,x)|=\limsup_{y \to x}\frac{|v(t,y)-v(t,x)|}{d(y,x)}.\]
 Recall that the positive and negative components, $[f(t,x)]_{+}$ and $[f(t,x)]_{-}$, of the function $f:(0,T)\times \Omega \to \mathbf{R}$ are, respectively, given by
\[[f(t,x)]_{+}=\max\{0,f(t,x)\},\quad [f(t,x)]_{-}=\max\{0,-f(t,x)\}.\]
\begin{dfn}
A function $\phi:(0,T)\times \Omega \to \mathbf{R}$ is said to be a subsolution test function if $\phi$ satisfies the following properties. 
\begin{enumerate}
\item
There exist local Lipschitz functions $\phi_{1},\phi_{2}:(0,T)\times \Omega \to \mathbf{R}$ satisfying $\phi={\phi}_{1}+{\phi}_{2}$.
\item
$|\nabla \phi_1(t,x)|=|\nabla^{-}\phi_1(t,x)|$ holds for all $(t,x) \in (0,T)\times \Omega$. Moreover, $|\nabla \phi_1|$ is continuous on $(0,T) \times \Omega$. 
\item
$\partial_{t}{\phi_1}, \partial_{t}{\phi_2}$ are continuous on $(0,T) \times \Omega$, respectively.
\end{enumerate}
Let $\underline{\mathcal{C}}$ denote the set of all subsolution test functions.

Similarly, $\phi :(0,T)\times \Omega \to \mathbf{R}$ is said to be a supersolution test function if $-\psi \in \mathcal{\underline{C}}$. Let $\overline{\mathcal{C}}$ denote the set of all supersolution test functions.
\end{dfn}
\begin{ex}
Consider $f(x)=|x|$ on $\mathbf{R}$. $|\nabla{f}(x)|=|\nabla^{-}{f}(x)|$ holds for $x\neq 0$ but not for $x=0$. This is because $|\nabla{f}(x)|=1,\,|\nabla^{-}{f}(x)|=0$ for $x=0$.
\end{ex}
Let $H$ be a Hamiltonian of the form
\begin{gather}H:\Omega\times [0,\infty) \to \mathbf{R},\quad
H(x,p)=\frac{1}{2}|p|^2+V(x) \label{gHamilt}
\end{gather}
 with $V \in C(\Omega)$.
In this section, we consider the Hamilton--Jacobi equation:
\begin{gather}
\partial_{t}u(t,x)+H(x,|\nabla u(t,x)|)=0,\,(t,x)\in (0,T)\times \Omega, \label{ghj1}\\
u(0,x)=g(x),\,x\in \Omega. \label{ghj2}
\end{gather}
For the function $f: (0,T) \times \Omega \to \mathbf{R}$, let us define the upper semicontinuous envelope, $f^{*}$, and the lower semicontinuous envelope, $f_{*}$, as follows:
\[f^{*}(t,x)=\limsup_{(s,y)\to (t,x)}f(s,y),\quad f_{*}(t,x)=\liminf_{(s,y) \to (t,x)}f(s,y).\]
\begin{dfn}\label{defvissol}
A locally bounded upper semicontinuous function $u :[0, T)\times \Omega \to \mathbf{R}$ is a viscosity subsolution of (\ref{ghj1}) and (\ref{ghj2}); 
if $u(0,x) \le g(x)$ on $\Omega$ and $u-\phi$ $(\phi \in \mathcal{\underline{C}})$ has a local maximum at $(t,x) \in (0,T)\times \Omega$, then
\[\partial_{t}{\phi}+H(x,|\nabla \phi_1(t,x)|-{|\nabla \phi_{2}(t,x)|}^\ast) \le 0.\]

A locally bounded lower semicontinuous function $u :[0,T)\times \Omega \to \mathbf{R}$ is a viscosity supersolution of (\ref{ghj1}) and (\ref{ghj2});
if $u(0,x) \ge g(x)$ on $\Omega$ and $u-\phi$ $(\phi \in \mathcal{\overline{C}})$ has a local minimum at $(t,x) \in (0,T)\times \Omega$, then
\[\partial_{t}{\phi}+H(x,|\nabla{\phi_1}(t,x)|+{|\nabla{\phi_{2}}(t,x)|}_\ast) \ge 0.\]
\end{dfn}
Next, we introduce the definition of a viscosity solution as described in \cite{N20??}. This definition is used at the end of the proof of Theorem \ref{main}. First, we define the class $\underline{\mathbf{C}}$ $(\overline{\mathbf{C}})$ of test functions.
\begin{dfn}
A function $\phi: (0,T) \times \Omega \to \mathbf{R}$ is said to be a subtest function if $\phi$ satisfies the following conditions. 
\begin{enumerate}
\item There exist local Lipschitz functions $\phi_{1}:\Omega \to \mathbf{R}$ $\phi_{2}: (0,T) \to \mathbf{R}$ such that $\phi=\phi_{1}(x)+\phi_{2}(t)$.
\item $|\nabla \phi_{1}|=|\nabla^{-} \phi_{1}|$ on $\Omega$ and $|\nabla^{-} \phi_{1}|$ is continuous on $\Omega$.
\item
$\phi_{2} \in C(0,T)$.
\end{enumerate}
Let $\underline{\mathbf{C}}$ denote the set of all subtest functions.
Similarly, a function $\phi:(0,T)\times\Omega \to \mathbf{R}$ is said to be a supertest function if $-\phi \in \underline{\mathbf{C}}$.
Let $\overline{\mathbf{C}}$ denote the set of all supertest functions. 
Furthermore, $C^{1,-}$ denotes all functions satisfying condition (2).
Likewise, $C^{1,+}$ denotes all functions satisfying condition (2) in terms of upper slope.
\end{dfn}
\begin{dfn}\label{dfnn}
A locally bounded upper semicontinuous function $u :[0, T)\times \Omega \to \mathbf{R}$ is said to be an s-viscosity subsolution of (\ref{ghj1}) and (\ref{ghj2}) if $ u(0,x) \le g(x)$ on $\Omega$ and $u-\phi$ $(\phi \in \underline{\mathbf{C}})$ has a local maximum at $(t,x)$,
\[\partial_{t}{\phi_{2}}(t)+H(x,|\nabla \phi_{1}(x)|) \le 0.\]

A locally bounded lower semicontinuous function $u :[0,T)\times \Omega \to \mathbf{R}$ is said to be an s-viscosity supersolution of (\ref{ghj1}) and (\ref{ghj2}) if $ u(0,x) \ge g(x)$ on $\Omega$ and $u-\phi$ $ (\phi \in \overline{\mathbf{C}})$ has a local minimum at $(t,x)$,
\[\partial_{t}{\phi_{2}}(t)+H(x,|\nabla \phi_{1}(x)|) \ge 0.\] 

A locally bounded continuous function $u:[0,T)\times \Omega \to \mathbf{R}$ is a viscosity solution of the Hamilton--Jacobi equations (\ref{ghj1}), (\ref{ghj2}) if $u$ is both an s-viscosity subsolution and an s-viscosity supersolution.
\end{dfn}
If a metric space $(\Omega,d)$ is locally compact, 
the definition presented in \cite{N20??} is equivalent to the aforementioned definition of viscosity solutions. The proofs of Proposition \ref{equivalence} and Lemma \ref{maximum prop} can be found in \cite{N20??}.
\begin{prop}\label{equivalence}
Let $(\Omega, d)$ be locally compact. Then, a viscosity subsolution of (\ref{ghj1}), (\ref{ghj2}) is an s-viscosity subsolution of (\ref{ghj1}), (\ref{ghj2}) if and only if an s-viscosity subsolution of (\ref{ghj1}), (\ref{ghj2}) is a viscosity subsolution of (\ref{ghj1}), (\ref{ghj2}).
This proposition holds for supersolutions.
\end{prop}
To prove Proposition \ref{equivalence}, the next lemma is required. The proof uses the definition of slope and the property of maximum value.
\begin{lem}\label{maximum prop}
    Let $u,v,w :\Omega \to \mathbf{R}$ be locally Lipschitz functions.
    \begin{enumerate}
        \item Suppose $v \in C^{1,-}$. If $u-v$ achieves a local maximum at $\hat{x} \in \Omega$, then $|\nabla u(\hat{x})| \ge |\nabla v(\hat{x})|$.
        \item
        Suppose $u \in C^{1,+}$. If $u-v$ achieves a local maximum at $\hat{x} \in \Omega$, then $|\nabla u(\hat{x})| \le |\nabla v(\hat{x})|$.
        \item Suppose $u \in C^{1,+}$ and $v \in C^{1,-}$. If $u-v+w$ achieves a local maximum at $\hat{x} \in \Omega$, then $||\nabla u(\hat{x})| - |\nabla v(\hat{x})||\le |\nabla w(\hat{x})|$.
    \end{enumerate}
\end{lem}

\begin{proof}[Proof of Proposition \rm{\ref{equivalence}}]
We only prove it for viscosity subsolutions. We do not mention the initial value condition because it is clear. If $u$ is an s-viscosity solution, it is clear that it is a viscosity solution; therefore, we want to show the reverse. We prove that $u$ is a viscosity subsolution and that $u-\phi$ and $\phi=\phi_{1}+\phi_{2} \in \underline{C}$ have a local maximum at $(\hat{t}, \hat{x}) \in (0,T) \times \Omega$. Consider a function $U$,
\[u(t,x) -\phi(t,y)-\frac{1}{2\epsilon} d(x,y)^{2}-\frac{\alpha}{2}d(\hat{x},y)^2-\frac{\alpha}{2}|t-\hat{t}|^2\],
where $\epsilon>0$, $\alpha>0$.
By local compactness, $U$ achieves a local maximum at some $(t_{\epsilon},x_{\epsilon},y_{\epsilon})$ in a compact neighborhood of $(\hat{t},\hat{x},\hat{x})$.
We can deduce that
\begin{align*}
&u(t_{\epsilon},x_{\epsilon}) -\phi(t_{\epsilon},y_{\epsilon})-\frac{1}{2\epsilon} d(x_{\epsilon},y_{\epsilon})^{2}-\frac{\alpha}{2}d(\hat{x},y_{\epsilon})^2-\frac{\alpha}{2}|t_{\epsilon}-\hat{t}|^2 \\
&\ge u(\hat{t},\hat{x})-\phi(\hat{t},\hat{x})\\
&\ge u(t_{\epsilon}, x_{\epsilon})-\phi(t_{\epsilon}, x_{\epsilon}).
\end{align*}
From this inequality, we obtain
\begin{equation}\label{evaluation1}
\frac{1}{2\epsilon} d(x_{\epsilon},y_{\epsilon})^{2}+\frac{\alpha}{2}d(\hat{x},y_{\epsilon})^2+\frac{\alpha}{2}|t_{\epsilon}-\hat{t}|^2 \le L d(x_{\epsilon},y_{\epsilon}),
\end{equation}
where $L$ is a Lipschitz constant of $\phi$.
By (\ref{evaluation1}), we obtain $(t_{\epsilon}, x_{\epsilon}, y_{\epsilon}) \to (t,x,x)$ as $\epsilon \to \infty$.
Because $u$ is a viscosity subsolution,
\[\partial_{t}\phi(t_{\epsilon}, y_{\epsilon})+\alpha (t_{\epsilon}-\hat{t})
+H(x_{\epsilon},\frac{1}{\epsilon}d(x_{\epsilon},y_{\epsilon})) \le 0.\]

Next, by fixing $x_{\epsilon}$ of $U$, we have
\[||\nabla \phi_{1}(t_{\epsilon},y_{\epsilon})|-\frac{1}{\epsilon}d(x_{\epsilon},y_{\epsilon})| \le |\nabla \phi_{2}(t_{\epsilon},y_{\epsilon})|+\alpha d(\hat{x}, y_{\epsilon}).\]
From our assumption of $H$, it follows that 
\begin{align*}
&\partial_{t}\phi(t_{\epsilon}, y_{\epsilon})+\alpha (t_{\epsilon}-\hat{t}) \\
&+H\left(x_{\epsilon}, |\nabla \phi_{1}(t_{\epsilon},y_{\epsilon})|-|\nabla \phi_{2}(t_{\epsilon},y_{\epsilon})|-\alpha d(\hat{x},y_{\epsilon})\right) \le 0.
\end{align*}
By letting $\epsilon \to 0$, we complete the proof.
\end{proof}
For (\ref{ghj1}) and (\ref{ghj2}), the comparison result is proved in \cite{GS2015}. In the following proposition, we do not assume the local compactness of $(\Omega,d)$.
\begin{prop}[{\cite[Theorem 4.2]{GS2015}}]\label{comparison} 
Let us assume that the viscosity subsolution and supersolution, $u$ and $v$, of {(\ref{ghj1}), (\ref{ghj2})} satisfy the following:
\begin{enumerate}
\item
For any bounded set $K$,
\begin{equation}
\lim_{t \to 0}\sup_{x \in K}([u(t,x)-g(x)]_{+}+[v(t,x)-g(x)]_{-})=0.\label{comp1}
\end{equation}
\item
The conditions,
\begin{gather}
\lim_{d(x,x_{0}) \to +\infty} \sup_{t \in [0,\infty)} \frac{u(t,x)}{1+d(x,x_{0})^2} \le 0, \notag \\
\lim_{d(x,x_{0}) \to +{\infty}}\sup_{t \in [0,\infty)} \frac{-v(t,x)}{1+d(x,x_{0})^2} \le 0. \label{comp2} 
\end{gather}
hold.
\end{enumerate}
Then, $u \le v$ on $(0,T) \times \Omega$.
\end{prop}
Some properties of the value function are investigated in \cite{GS2014,GS2015}.
\begin{prop}[{\cite[Theorem 4.8]{GS2015}}]
Let us assume that the initial function $g:\Omega \to \mathbf{R}$ of (\ref{ghj1}), (\ref{ghj2}) is continuous. Then, for $(t,x) \in (0,T) \times \Omega$, a value function,
\begin{equation*}
u(t,x)=\inf_{\sigma}\left\{ \int_{0}^{t}L(x,|\nabla\sigma'|)ds+g(\sigma(0)) :\sigma(t)=x \right\}, \label{value}
\end{equation*}
is a viscosity solution of (\ref{ghj1}) and (\ref{ghj2}), where $L: \Omega \times [0,\infty) \to \mathbf{R}$ is defined by $L(x,v)=\frac{1}{2}v^2-V(x)$ and $\sigma: [0,t] \to \Omega$ moves over the class of absolutely continuous functions. The metric derivative $|\nabla \sigma'(t)|$ is defined as follows:
\begin{equation*}
\lim_{h \to 0} \frac{d(\sigma(t+h),\sigma(t))}{|h|}.
\end{equation*}
An absolutely continuous curve in a metric space is defined to be a curve that admits metric derivatives almost everywhere (cf. \cite{B2020}).
\end{prop}
Finally, we prove a proposition used in the proof of Theorem \ref{main}. We assure the local compactness of $(\Omega, d)$. 
\begin{prop}\label{equilem}
Let $u: (0,T) \times \Omega \to \mathbf{R}$ be an upper semicontinuous  function.
Conditions (1) and (2) are equivalent if the following hold.
\begin{enumerate}
    \item
    $u-\phi$ $(\phi \in \underline{\mathbf{C}})$ has a local maximum at $(t,x)$,
     \[\partial_{t}{\phi_{2}}(t)+H(\bar{x},|\nabla \phi_{1}(x)|) \le 0.\]
    \item
    For $\hat{a}, \hat{x} \in \Omega$ and $\hat{t} \in (0,T)$, if $u(t,x)-\phi(t)-\frac{l}{2}d^{2}(\hat{a},x)$ takes a local maximum at $(\hat{t},\hat{x})$, then
    \[\partial_{t}\phi(\hat{t})+H(\hat{x},l d(\hat{a},\hat{x}))\le 0,\]
    where $\phi \in C^{1}(0,T)$ and $l>0$.
\end{enumerate}
\end{prop}
\begin{proof}
Basically, the proof is inspired by \cite{N20??}. When the condition of (2) is assured, we can prove the conclusion of (1). Let $\phi \in \underline{\mathbf{C}}$. Consider a function:
\[u(t,x)-\phi(t,y)-\frac{1}{2\epsilon}d(x,y)^2-\frac{\alpha}{2}d(\hat{x},y)^2-\frac{\alpha}{2}|t-\hat{t}|^2.\]
The rest can be done as in the proof of Proposition \ref{equivalence}. The opposite is clear.
\end{proof}
\section{Proof of Theorem \ref{main}}
Before proceeding with the proof of Theorem \ref{main}, we derive a uniform Lipschitz estimate of the viscosity solution $u_{n}$ of (\ref{Phj1}), (\ref{Phj2}). 
Basically, the Lipschitz estimate of the viscosity solution of the Hamilton--Jacobi equation is identical to that in the case of Euclidean space (cf. \cite{ABIL2013}) if we use the basic properties of viscosity solutions in metric space (cf. \cite{GS2014, GS2015}). For the sake of the complexity of the setting and the self-containedness of the paper, we include the proof of the Lipschitz estimate.
Note that the viscosity solutions $u_{n}$ of (\ref{Phj1}) and (\ref{Phj2}) are uniquely given in the form of value functions using Proposition \ref{comparison}.
\begin{prop}\label{lip}
Fix $m \in \mathbf{N}_{\ge 0}$. For $n \ge m$, the viscosity solution $u_{n}$ of (\ref{Phj1}) and (\ref{Phj2}) is uniformly Lipschitz continuous over $[0, T) \times \mathcal{N}^n$. The estimate:
\begin{equation}\label{lipschitz}
|u_{n}(t,y)-u_{n}(t,x)|\le 2K\delta_{n}(x,y)
\end{equation}
holds for $x,y \in \mathcal{N}^{n}$,
where $K$ is a positive constant independent of $n$.
In particular, $u_{n}$ is uniformly Lipschitz continuous on $[0, T) \times \mathcal{N}^m$.
\end{prop}
\begin{proof}
For $h>0$, it follows from Theorem \ref{comparison} that
\begin{equation}
|u_{n}(x,t+h)-u_{n}(x,t)| \le |u_{n}(x,h)-u_{n}(x,0)|,(t,x)\in (0,T)\times \mathcal{N}^n. \label{pflip1}
\end{equation}
For $C>0$, define functions on $\mathcal{N}^n$ as follows:
\[\underline{u}(t,x)=-Ct+g_{n}(x),\,\overline{u}=Ct+g_{n}(x). \]
We show that they are the viscosity subsolution and supersolution of (\ref{Phj1}) and (\ref{Phj2}), respectively, for a large $C$. We only confirm that $\underline{u}$ is a viscosity subsolution. Indeed, for $\psi \in \underline{\mathcal{C}}$, we assume that $\underline{u}-\psi$ admits a maximum value at $(t,x)\in (0,T)\times \mathcal{N}^n$. Because $\psi=\psi_{1}+\psi_{2}$,
\begin{equation*}
\psi_{1}(t,y)-\psi_{1}(t,x)\ge \underline{u}(t,y)-\psi_{2}(t,y)-\left(\underline{u}(t,x)-\psi_{2}(t,x)\right).
\end{equation*}
By this equation and the definition of slope, we obtain
\begin{align*}
&|\nabla{\psi_{1}}(t,x)|=|\nabla^{-}\psi_{1}(t,x)| \notag \\
&\le |\nabla{(\underline{u}-\psi_{2})}(t,x)| \notag \\
&\le |\nabla{\underline{u}}(t,x)|+|\nabla{\psi_{2}}(t,x)| \le C_{1}+{|\nabla{\psi_{2}}(t,x)|}^{\ast},
\end{align*}
where $C_{1}$ is the Lipschitz constant of $g_{n}$. Thus, by the monotonicity and continuity of the Hamiltonian, there exists a constant $C_{2}>0$ such that 
\begin{equation*}
H_{n}(x,|\nabla{\psi_1}(t,x)|-{|\nabla{\psi_{2}}(t,x)|}^\ast)\le \sup_{0\le r \le C_{1},x \in \mathcal{N}^n}H(x,r) \le C_{2} <\infty.
\end{equation*}
If we take $C=C_{2}$, we have
\[\partial{\psi_{t}}+H_{n}(x,|\nabla{\psi_1}(t,x)|-{|\nabla{\psi_{2}}(t,x)|}^\ast) \le -C_{2}+C_{2}=0.\]
Now, this proves that $\underline{u}$ is a viscosity subsolution. Furthermore, by the comparison theorem, $u_{n}$ coincides with the viscosity solution produced by Perron's method. Thus,
\begin{equation}
-Ch+g_{n}(x)\le u_{n}(h,x) \le Ch+g_{n}(x). \label{pflip2}
\end{equation}
By (\ref{pflip1}) and (\ref{pflip2}), Lipschitz continuity for $t$ is obtained.

Consider arbitrary $K>0,\alpha>0,\epsilon>0$ and fix $(t,x)\in (0,T)\times \mathcal{N}^n$. Define $\Psi: [0,T] \times \mathcal{N}^n \to \mathbf{R}$ by
\[\Psi(s,y)=u_{n}(s,y)-u_{n}(t,x)-2K\sqrt{\delta_{n}(x,y)^{2}+\epsilon^2}-\frac{|t-s|^2}{\alpha^2}. \]
This function admits a maximum at an interior point if $\alpha$ is sufficiently small. Assume $K$ does not depend on $n$ s.t. $H_{n}(x,p) >2C$ holds if $|p| \ge K, x \in \mathcal{N}^{n}$. For example, take $K=\sqrt{6C+2\max_{\mathcal{N}} V^{+}}$. Let $(\bar{s}_{\epsilon, \alpha}, \bar{y}_{\epsilon, \alpha})$ be the maximum point. Because $u_{n}$ is a viscosity subsolution, we obtain
\begin{equation}                    
2\frac{(t-\bar{s}_{\epsilon,\alpha})}{\alpha^2}+H_{n}(\bar{y}_{\epsilon, \alpha}, 2K\frac{\delta_{n}(x,\bar{y}_{\epsilon, \alpha})}{\sqrt{\delta_{n}(x,\bar{y}_{\epsilon, \alpha})^2+\epsilon^2}})\le0.\label{pflip3}
\end{equation}
However, by $\Psi(\bar{y}_{\epsilon, \alpha}, t)\le \Psi(\bar{y}_{\epsilon, \alpha},\bar{s}_{\epsilon, \alpha})$ and the Lipschitz continuity for $t$, the following equation holds:
\begin{equation}
|2\frac{(t-\bar{s}_{\epsilon, \alpha})}{\alpha^2}|\le 2C.\label{pflip4}
\end{equation}
From (\ref{pflip3}) and (\ref{pflip4}), we have
\[H_{n}(\bar{y}_{\epsilon, \alpha}, 2K\frac{\delta_{n}(x,\bar{y}_{\epsilon, \alpha})}{\sqrt{\delta_{n}(x,\bar{y}_{\epsilon, \alpha})^2+\epsilon^2}})\le 2C.\]
Considering the construction of $K$, it follows that
\[2K\frac{\delta_{n}(x,\bar{y}_{\epsilon, \alpha})}{\sqrt{\delta_{n}(x,\bar{y}_{\epsilon, \alpha})^2+\epsilon^2}} \le K.\]
By letting $\epsilon \to 0$, we obtain $\bar{y}_{\epsilon, \alpha} \to {x}$.
Conversely, $\bar{s}_{\epsilon, \alpha}=t$ 
.
Now, we have 
\[\Psi(t,y)\le\Psi(t,\bar{y}_{\epsilon, \alpha}).\]
This implies that, as $\epsilon \to 0$,
\begin{equation*}
u_{n}(t,y)-u_{n}(t,x)-2K\delta_{n}(x,y)\le0.
\end{equation*}
If we now replace $x$ and $y$ and repeat the same argument, we obtain Lipschitz continuity on $\mathcal{N}^n$. In particular, by restricting to $\mathcal{N}^m$,
$\delta_{n}(x,y)\le \delta_{m}(x,y)$ 
holds for $x,y \in \mathcal{N}^{m}$. Hence, Lipschitz continuity on $(\mathcal{N}^m,\delta_{m})$ is obtained.
\end{proof}
\begin{rem}\label{comm.lip}
If we admit condition (2) of Theorem \ref{main}, Lipschitz estimate (\ref{lipschitz}) holds for $\tilde{d}$. 
\end{rem}
\begin{prop}\label{boundness}
Fix $m \in \mathbf{N}_{\ge 0}$ and $n \ge m$. The viscosity solution $u_{n}$ of (\ref{Phj1}), (\ref{Phj2}) is uniformly bounded on $[0,T) \times \mathcal{N}^m$.
\end{prop}
\begin{proof}
This is clear from (\ref{pflip2}).
\end{proof}
In the sequel, we extend the range of the validity of uniform Lipschitz continuity and uniform boundedness to $t=T$ to enable the application of the Ascoli-Arzel\`{a} theorem.
\begin{prop}\label{prep}
Fix $m \in \mathbf{N}_{\ge 0}$ and $n \ge m$. 
Then, $u_{n}$ is monotonically decreasing on $[0,T] \times \mathcal{N}^{m}$ with respect to $n$, and
there exists ${u} \in C([0,T] \times \mathcal{N})$ such that $u_{n}$ converges to ${u}$ uniformly on $[0,T] \times \mathcal{N}^{m}$.
\end{prop}
\begin{proof}
Based on Propositions \ref{lip} and \ref{boundness}, we apply the Ascoli-Arzel\`{a} theorem to $u_{n}$ on $[0,T] \times \mathcal{N}^m$. Then, there exists a Lipschitz function $u:[0,T] \times \mathcal{N}^m \to \mathbf{R}$ such that there exists a subsequence, $u_{n}$, that converges to $u$ uniformly. From the value function representation of $u_{n}$, it is evident that $u_{n}$ is monotonically decreasing on $[0,T] \times \mathcal{N}^m$. It follows that $u_{n}$ converges to $u$ uniformly irrespective of the subsequence. Because $m$ is arbitrarily fixed, ${u}$ is continuous on $[0,T]\times \bigcup_{l=0}^{\infty}\mathcal{N}^{l}$. Then, using the uniform Lipschitz continuity, we should naturally extend $u$ continuously up to $\mathcal{N}$. See Remark \ref{comm.lip}.
\end{proof}
We now provide a proof of Theorem \ref{main}.
\begin{proof}[Proof of Theorem \rm{\ref{main}}]
Let $u$ be defined as in Proposition \ref{prep}.
In view of Proposition \ref{equilem}, it is sufficient to prove that ${u}$ is an s-viscosity solution. Because we can prove that $u$ is an s-viscosity supersolution by a similar argument, we only prove that ${u}$ is an s-viscosity subsolution. The initial condition is satisfied by the form of the value function. Assume that $u-\psi,\psi \in \underline{\mathbf{C}},$ admits a local maximum at $(\overline{t},\overline{x})\in (0,T)\times \mathcal{N}$. Note that \[\mathcal{N}=\overline{\bigcup_{l=0}^{\infty}\mathcal{N}^{l}}.\] First, consider the case $\bar{x} \in \bigcup_{l=0}^{\infty}\mathcal{N}^{l}$. Then, for $\overline{x}$, there exists some $m \in \mathbf{N}_{\ge 0} $ such that $\overline{x} \in {V}^{m}$ or $\overline{x} \in{E}^m$, which is not a vertex of any $\mathcal{N}^{n}$. Because the latter case can be similarly discussed, we only consider the former case. Let $r$ be in (2) of Definition \ref{expanding}. Consider the following domain. For $\kappa>0$ and $n \ge m$, define
\[B_{n}(\overline{t},\overline{x})=B_{\kappa}(\overline{t})\times \left(B_{r}(\overline{x})\cap \mathcal{N}^n\right). \]
Let us define $D_{n}(\overline{x})=B_{r}(\overline{x})\cap \mathcal{N}^n$. Furthermore, let us take $\kappa$, $r$ to be sufficiently small such that ${u}-\psi$ admits a maximum value on $B_{n}(\bar{t},\bar{x})$. 
Let $C$ be the common Lipschitz constant for $u_{n},{u}$ on $\mathcal{N}^m$. For sufficiently small $\epsilon>0$, let us define
\[\Theta_{n}=\sup_{{B}_{n}}\left\{(u_{n}-{u})(t,x)-\eta_{n}^{1}(x)-\eta_{n}^{2}(t)\right\},\]
\[\eta_{n}^{1}(x)=\left(2C+\frac{1}{n}\right)\sqrt{{\delta_{n}(x,\overline{x})}^2+\epsilon},\,\eta_{n}^{2}(t)=\left(2C+\frac{1}{n}\right)\sqrt{(t-\overline{t})^2+\epsilon}. \]
By constructing $\Theta_{n}$ in this manner, $\Theta_{n} \to 0, n \to \infty$ by the uniform convergence of $u_{n}$.
Furthermore, let us define test functions as follows:
\[\psi_{1}^{n}=\psi_{1}(x)+\eta_{n}^{1}(x),\quad \psi_{2}^{n}=\psi_{2}(t)+\eta_{n}^{2}(t).\]
Then, $u_{n}-(\psi_{1}^{n}+\psi_{2}^{n})$ has a maximum at some point $(\hat{t}_{n},\hat{x}_{n})\in {B}_{n}$ and $(\hat{t}_{n},\hat{x}_{n})\to (\bar{t},\bar{x})$ as $n \to \infty$. Note that
\begin{align*}
u_{n}-(\psi_{1}^{n}+\psi_{2}^{n})=
\left(u_{n}-{u}-\eta_{n}^{1}-\eta_{n}^{2}\right)+\left(u-\psi_{1,k(n)}-\psi_{2}\right).
\end{align*}
The left-hand side converges uniformly on $B_{m}(\bar{t},\bar{x})$ to the function that has a maximum value at $(\bar{t},\bar{x})$ by the discussions for $\Theta_{n}$. 
Because $u_{n}$ is an s-viscosity subsolution, we have
\begin{equation}\label{pfmain1}
\partial_{t}\psi_{2}(\hat{t}_{n})+\frac{\hat{t}_{n}-\overline{t}}{\sqrt{(\hat{t}_{n}-\overline{t})^2+\epsilon}}+H_{n}(\hat{x}_{n}, |\nabla \psi_{1}^{n}(\hat{x}_{n})|)\le0.
\end{equation}
By a simple calculation, we obtain
\begin{equation}\label{pfmain2}
|\nabla \psi_{1}^{n}(\hat{x}_{n})| \ge |\nabla \psi_{1}(\hat{x}_{n})|-\frac{\delta_{n}(\hat{x}_{n},\bar{x})}{\sqrt{\delta_{n}(\hat{x}_{n},\bar{x})^2+\epsilon}}.
\end{equation}
By combining (\ref{pfmain1}), (\ref{pfmain2}), and the monotonicity of $H$, we obtain
\begin{align}\label{pfmain3}
&\partial_{t}\psi_{2}(\hat{t}_{n})+\frac{\hat{t}_{n}-\overline{t}}{\sqrt{(\hat{t}_{n}-\overline{t})^2+\epsilon}} \notag\\
&+H_{n}(|\hat{x}_{n}, |(\nabla \psi_{1}(\hat{x}_{n})|-\frac{\delta_{n}(\hat{x}_{n},\bar{x})}{\sqrt{\delta_{n}(\hat{x}_{n},\bar{x})^2+\epsilon}})\le0.
\end{align}
By letting $n \to \infty$, we obtain
\begin{equation}\label{case1}
\partial_{t}{\psi_{2}}(\overline{t})+H(\overline{x},|\nabla \psi_{1}(\bar{x})|)\le0.
\end{equation}
Next, consider the case of $\bar{x} \in \overline{\bigcup_{l=0}^{\infty}\mathcal{N}^{l}}\setminus
\bigcup_{l=0}^{\infty}\mathcal{N}^{l}$.
By Proposition \ref{equilem},
it is sufficient to show that, for $\bar{t}$, $\bar{x}$, and any fixed $\bar{a} \in \mathcal{N}$, if
\[u(t,x)-\phi(t)-\frac{l}{2}\tilde{d}^{2}(\bar{a},x)\]
achieve a local maximum at $(\bar{t}, \bar{x})$, then
\[\partial_{t}\phi(\bar{t})+H(\bar{x}, l \tilde{d}(\bar{a}, \bar{x})) \le 0.\]
We choose a sequence $\{a_{n}\}_{n \in \mathbf{N}_{\ge 0}}$ s.t. $a_{n} \in \mathcal{N}^{n}$ and $a_{n} \to \bar{a}, n \to \infty$. Consider a function $F_{n}(t,x):((0,T) \times \mathcal{N}^{n})\cap G \to \mathbf{R}$,
\[u_{n} (t,x)-\phi(t)-\frac{l}{2}{\delta_{n}}^{2}(a_{n},x),\]
where $G$ is a compact neighborhood of $(\bar{t}, \bar{x})$ in $\mathcal{N}$. Let $n$ be sufficiently large.
The function $F_{n}$ has a maximum at some $(t_{n},x_{n})$ and $(t_{n},x_{n}) \to (\bar{t}, \bar{x})$ with respect to $\tilde{d}$ by assumptions (2) and (3) in Theorem \ref{main}, and if necessary, take a subsequence. Because $u_{n}$ is an s- viscosity subsolution on $\mathcal{N}^{n}$,
\[\partial_{t}\phi(t_{n})+H_{n}(x_{n},l \delta_{n}(a_{n}, x_{n})) \le 0.\]
By letting $n \to \infty$, we obtain a conclusion by assumption (2) in Theorem \ref{main} and the continuity of $H$.
By similarly constructing a test function, we can prove that ${u}$ is an s-viscosity supersolution.
\end{proof}
\begin{rem}
In the proof of Theorem \ref{main}, we can similarly derive (\ref{case1}) when discussing case $\bar{x} \in \overline{\bigcup_{l=0}^{\infty}\mathcal{N}^{l}}\setminus
\bigcup_{l=0}^{\infty}\mathcal{N}^{l}$.
This allows us to show that Theorem \ref{main} holds even if we omit definition (2) of expanding networks. However, for application purposes, it is sufficient to weaken condition (2) regarding $r$ in Definition \ref{expanding} to $r_{l}$.
\end{rem}
\section{Discussions and further problems}
In the previous section, we saw that the first half of the proof does not work by simply approximating the test function when considering the Sierpinski gasket. We then present possible problems.

Now, let us consider part $\bar{x} \in \bigcup_{l=0}^{\infty}\mathcal{N}^{l}$ in the proof of Theorem \ref{main} using the example of the Sierpinski gasket. That is, let $\{\mathcal{N}^{n}\}_{n \in \mathbf{N}_{\ge 0}}$ be prefractals and $\mathcal{N}$ be the Sierpinski gasket. See Example \ref{sierpinski}. In this case, we replace neighborhood $B_{n}(\bar{t},\bar{x})$ in the proof of Theorem \ref{main} with $B_{n}(\bar{t},\bar{x})=B_{\kappa}(\overline{t})\times \left\{B_{r_{n+1}}(\overline{x})\cap \mathcal{N}^n\right\}$. Here, we define $r_{n}=\frac{1}{2^{n+1}}$.
Because the Sierpinski gasket has dense branching points, it is not evident that the regularity condition of the test function is satisfied only because the test function on the Sierpinski gasket is restricted to an open subset of the prefractal in general. This makes it necessary to approximate the test function. A function with regularity can be approximated as follows.

\begin{lem}\label{app}
Let $\psi_{1}$, and $\bar{x}$ be as defined in the proof of Theorem \ref{main}. Then, for $\psi_{1}$, there exists $\psi_{1,k}: B_{n}(\bar{t},\bar{x})\to \mathbf{R}$ such that
$\psi_{1,k}$ converges to $\psi_{1}$ uniformly on $B_{n}(\bar{t},\bar{x})$ and satisfies the condition of a subtest function on $\mathcal{N}^{n}$.
\end{lem}
\begin{proof}
For simplicity, let us assume $\psi_{1}(\bar{x})=0$. Because $\psi_{1}$ is a function concerning spatial variables, it is sufficient to construct $\psi_{1,k}$ in $D_{n}(\bar{x})$. We construct the test function sequentially. Let $N$ denote the degree of the graph at $\bar{x}$.
\begin{enumerate}
\item 
We identify $D_{n}(\bar{x})$ as
\[D_{n}(\bar{x})=\bigcup_{i=1}^{N}L_{i},\]
where $L_{i}=\{ \tau_{i}: 0\le \tau_{i} <r_{n+1} \subset \mathbf{R}_{\ge 0}$, $i \in \{1,\cdots,N\}$, and $\bar{x}$ correspond to the origin $0$ of each $L_{i}$. Furthermore, define $\psi^{i}:L_{i} \to \mathbf{R}$ by $\psi^{i}=\psi|_{L_{i}}$ for each $L_{i}$.
\item
Let $k \ge n+2$. Let $M_{i,k} \subset L_{i}$ be $M_{i,k}=[0,r_{k}]$.
Let us define $O_{k} \subset D_{n}(\bar{x})$ by 
\[O_{k}=\bigcup_{i=1}^{N}M_{i,k}.\]
 Let us now define an affine function $l: D_{n}(\bar{x})\to \mathbf{R}$ that satisfies $l(\bar{x})=0$ on $D_{n}(\bar{x})$ and the condition for a subtest function. For example, we can take the affine function $l^{i}:L_{i} \to \mathbf{R}$ that satisfies the following conditions:
\begin{equation*}
\quad (l^{1})'(+0)\ge 0,\quad (l^{2})'(+0)=(l^{3})'(+0)=(l^{4})'(+0)=-(l^{1})'(+0),
\end{equation*}
\begin{equation*}
l^{1}(0)=l^{2}(0)=l^{3}(0)=l^{4}(0).
\end{equation*}
Define $l:D_{n}(\bar{x}) \to \mathbf{R}$ by
\begin{equation*}
l(x)=l^{i}(x),\quad x \in L_{i}.
\end{equation*}
\item
On each $L_{i}$, define $\tilde{\psi}_{i} :L_{i} \to \mathbf{R}$ by
\begin{equation*}
\tilde{\psi}_{i}=
\begin{cases}
l^{i}(x) & (x \in M_{i,k}),\\
\psi^{i}(x-r_{k})+l^{i}(r_{k})- \psi^{i}(0) & (x \in L_{i} \setminus M_{i,k}).
\end{cases}
\end{equation*}
Using the mollifier $\rho_{k}:=\rho_{r_{k+1}}$, let us define $\psi_{1,k}: D_{n}(\bar{x}) \to \mathbf{R}$ by
\[\psi_{1,k}(x)=\rho_{k}\ast \tilde{\psi}_{i}(x),\, x \in L_{i}. \]
When applying the mollifier, we extend $\tilde{\psi_{i}}$ to $\mathbf{R}$ appropriately.
\end{enumerate}
The sequence of functions $\psi_{1,k}$ converges at each point on $L_{i}$. Let $x \in L_{i}$. We observe
\begin{align}\label{eq1}
|(\psi_{1,k}-\psi^{i})(x)|&= 
\left|\int_{\mathbf{R}} \rho_{k}(y)\tilde{\psi}_{i}(x-y)dy
-\int_{\mathbf{R}}\rho_{k}(y)\psi^{i}(x)dy\right|\notag\\
 &\le
\int_{\mathbf{R}}\rho_{k}(y)|\tilde{\psi}_{i}(x-y)-\psi^{i}(x)|dy\notag\\
 &=
\int_{[-r_{k+1},r_{k+1}]}\rho_{k}(y)|\tilde{\psi}_{i}(x-y)-\psi^{i}(x)|dy.
\end{align}
From $-r_{k+1} \le x-y \le r_{n+1}+r_{k+1}$, we see that
\begin{align*}
&|\tilde{\psi}_{i}(x-y)-\psi^{i}(x)|\\
&\le
\begin{cases}
|l^{i}(x-y)-\psi^{i}(x)|, & (-r_{k+1} \le x-y \le r_{k}).\\
|\psi^{i}(x-\frac{1}{2^{k}})+l^{i}(\frac{1}{2^{k}})-\psi^{i}(0)-\psi^{i}(x)|, & (r_{k} < x-y \le r_{n+1}+r_{k+1}).
\end{cases}
\end{align*}
For $x=0$, we have
\begin{align}\label{eq2}
|\tilde{\psi}_{i}(0-y)-\psi^{i}(0)|&=
|l^{i}(-y)-\psi^{i}(0)|\notag\\
&=
|l^{i}(-y)-l^{i}(0)|+|l^{i}(0)-\psi^{i}(0)|\notag\\
&\le
C^{i}|y|+|l^{i}(0)-\psi^{i}(0)|\notag\\
&\le
C^{i}\cdot r_{k+1},
\end{align}
where $C^{i}$ denotes the Lipschitz constant of $l^{i}$ in $L_{i}$.
For $x\neq 0$, taking sufficiently large $k$ such that $r_{k}<x-y$ for $y \in [-r_{k+1},r_{k+1}]$, we obtain
\begin{align}\label{eq3}
|\tilde{\psi}_{i}(x-y)-\psi^{i}(x)|&=
|\psi^{i}(x-r_{k})+l^{i}(r_{k})-\psi^{i}(0)-\psi^{i}(x)|\notag\\
&\le
C\cdot r_{k}+|l^{i}(r_{k})-\psi^{i}(0)|.
\end{align}
Combining (\ref{eq1}), (\ref{eq2}), and (\ref{eq3}), the pointwise convergence of $\rho_{k}\ast \tilde{\psi}_{i}(x)$ is proved.
By applying the Ascoli-Arzel\'{a} theorem to $\rho_{k}\ast \tilde{\psi}_{i}(x)$, we obtain a uniformly convergent subsequence. Note that the Lipschitz constant of $\tilde{\psi}_{i}(x)$ is independent of $k$. As a result, we obtain a uniformly convergent subsequence of $\rho_{k}\ast \tilde{\psi}_{i}(x)$. We denote it by $\rho_{k}\ast \tilde{\psi}(x)$. When the operation is repeated by taking a subsequence for each $L_{i}$, $\psi_{1,k}$ converges uniformly on $D_{n}$. Because we use the mollifier, $\psi_{1,k}$ satisfies the condition of a subtest function not only at $\bar{x}$ but also on the entire $D_{n}(\bar{x})$.
\end{proof}
By using $\psi_{1,k}$ in Lemma \ref{app}, $\psi_{1}^{n}$ in the proof of Theorem \ref{main} is replaced as follows:
\[\psi_{1}^{n}=\psi_{1,k(n)}(x)+\eta_{n}^{1}(x)+\Theta_{n,k(n)}.\]
Note that we replace $k$ with $k(n)$ to control a maximizer of $\psi_{1,k}$.
A similar method of calculation leaves $|\nabla \psi_{1,k(n)}|$ in the contents of the Hamiltonian in the inequality that corresponds to (\ref{pfmain3}). Unfortunately, the previous approximation does not guarantee that the slope of the function also converges at the same time.
I would therefore like to pose the next question:
\begin{itemize}
    \item In a metric space, provide a method of approximating the function such that the regularity and convergence of the slope is guaranteed.
\end{itemize}

Now let us consider an application of Theorem \ref{main}.
Our original motivation is to adapt Theorem \ref{main} to fractals. However, this comes with certain difficulties. It is not easy to check whether a fractal satisfies assumption (2) of Theorem \ref{main}. In this sense, I would like to present the following problem:
\begin{itemize}
\item What kind of fractals and their approximate sequences satisfy assumption (2) of Theorem \ref{main}?
\end{itemize}
We believe that it holds at least in the case of the Sierpinski gasket.

Finally, we would like to mention a variant of Theorem \ref{main}. It is an interesting proposition to consider spatial stability with respect to other equations, such as elliptic equations. For example, a related study is \cite{H2018}. It is worth mentioning that in viscosity solution theory, as far as we know, a formulation for the solution of the second-order partial differential equation in a general metric space does not exist. Thus, problems still exist regarding the topic of spatial stability.
\section{Acknowledgments}
The author would like to express sincerest gratitude to Professor Nao Hamamuki and Atsushi Nakayasu.
Mr. Keisuke Abiko, Taiga Kurokawa, and Ryo Negishi also provided insightful comments and suggestions.
This work was supported by JST SPRING, Grant Number JPMJSP2119.

\end{document}